\documentclass[11pt]{amsart}
\usepackage{amsmath, amsthm, amsfonts, amssymb}
\usepackage[all]{xy}

\newcommand{\PP}{\mathbb{P}}
\newcommand{\C}{\mathbb{C}}
\newcommand{\M}{\mathcal{M}}
\newcommand{\CC}{\mathcal{C}}
\newcommand{\K}{\mathbb{K}}

\newcommand{\Kvars}{\K[x_1,\ldots,x_n]}

\DeclareMathOperator{\rank}{rank}

\newtheorem{thm}{Theorem}[section]
\newtheorem{defn}[thm]{Definition}
\newtheorem{prop}[thm]{Proposition}
\newtheorem{lemma}[thm]{Lemma}

\newtheorem{remark}[thm]{Remark}
\newtheorem{example}[thm]{Example}

\newcommand{\doslineas}[2]{\genfrac{}{}{0pt}{}{#1}{#2}}
\title[Interpolation of Ideals]{Interpolation of Ideals}
\author[M. Avendano]{Martin Avendano}
\address{Centro Universitario de la Defensa-IUMA\\ 
Academia General
Militar\\ 
Ctra. de Huesca s/n\\ 
50090, Zaragoza, Spain} 
\email{avendano@unizar.es, jortigas@unizar.es}

\author[J. Ortigas]{Jorge Ortigas-Galindo$^1$}
\address{} 
\email{}

\begin{document}

\thanks{Second author is partially supported by the Spanish Ministry of Education MTM2010-21740-C02-02.}  
\footnotetext[1]{Corresponding author. Email: jortigas@unizar.es Phone: +34 976739636 Fax: +34 976739824.}

\begin{abstract}
Let $\K$ denote an algebraically closed field.
We study the relation between an ideal $I\subseteq\Kvars$ and its cross
sections $I_\alpha=I+\langle x_1-\alpha\rangle$. In particular, we study under
what conditions $I$ can be recovered from the set $I_S=\{(\alpha,I_\alpha)\,:\,\alpha\in S\}$
with $S\subseteq\K$. For instance, we show that an ideal $I=\bigcap_iQ_i$, where $Q_i$ is
primary and $Q_i\cap\K[x_1]=\{0\}$, is uniquely determined by $I_S$ when $|S|=\infty$.
Moreover, there exists a function $B(\delta,n)$ such that, if $I$ is generated by polynomials
of degree at most~$\delta$, then $I$ is uniquely determined by $I_S$ when $|S|\geq B(\delta,n)$.
If $I$ is also known to be principal, the reconstruction can be done when $|S|\geq 2\delta$,
and in this case, we prove that the bound is sharp.
\end{abstract}

\maketitle

\section{Introduction}

Throughout this paper $\K$ will be an algebraically closed field.
The main result in the theory of univariate polynomial interpolation states that
for any given $d+1$ points $\{(\alpha_i,\beta_i)\in \K^2\;:\,i=1,\ldots,d+1\}$ there
exists a unique polynomial $p\in \K[x]$ of degree bounded by $d$ such that $f(\alpha_i)=\beta_i$
for $i=1,\ldots,d+1$. The uniqueness part of this statement says that a planar
algebraic curve $\mathcal{C}=\{(x,y)\in \K^2\;:\;y=p(x)\}$ of degree bounded by $d$ is uniquely
determined by its intersection with $d+1$ parallel lines $\{x=\alpha_i\}$. In this paper we
study generalizations of this fact to higher dimensions, i.e. we study under what conditions
it is possible to recover an algebraic variety $V\subseteq \K^n$ from its intersection with parallel
hyperplanes. We also consider the algebraic
counterpart of the problem, i.e. under what conditions it is possible to recover an ideal
$I\subseteq \Kvars$ from some cross sections $I+\langle x_1-\alpha\rangle$. Our first result
studies the simplest situation of all, that is, when all the cross sections are known:

\begin{thm}\label{thm-1}
Let $I\subseteq \Kvars$ be an ideal. Then:
\begin{itemize}
 \item[(a)] $\displaystyle\sqrt{I}=\bigcap_{\alpha\in\K}\sqrt{I+\langle x_1-\alpha\rangle}$.
 \item[(b)] $\displaystyle I=\bigcap_{\alpha\in \K}\bigcap_{k\geq 1}I+\langle x_1-\alpha\rangle^k$.
\end{itemize}
\end{thm}

For radical ideals $I\subseteq \Kvars$, Theorem~\ref{thm-1}(a) implies that
$I=\bigcap_{\alpha\in \K}I+\langle x_1-\alpha\rangle$, since:
\begin{equation}\label{eq_all}
I\subseteq\bigcap_{\alpha\in \K}I+\langle x_1-\alpha\rangle=\bigcap_{\alpha\in \K}\sqrt{I}+\langle x_1-\alpha\rangle
\subseteq\bigcap_{\alpha\in \K}\sqrt{I+\langle x_1-\alpha\rangle}\;\overset{\ref{thm-1}(a)}{=}\;\sqrt{I}=I. 
\end{equation}
However, this reconstruction formula is not valid for general ideals: $I=\langle xy\rangle$
and $J=\langle x^2y\rangle$ are two different ideals of $\K[x,y]$ that have exactly
the same cross sections $I+\langle x-\alpha\rangle=J+\langle x-\alpha\rangle$ for all $\alpha\in \K$.
Theorem~\ref{thm-1}(b) shows that this problem can be avoided by including powers of the ideals
$\langle x_1-\alpha\rangle$. Informally speaking, these powers account for the multiplicities hidden
in $I$ that are not visible geometrically in $V(I)$.

\bigskip

Our second result corresponds with the situation where infinitely many cross sections are known, i.e.
the problem of recovering an ideal $I$ from the set $I_S=\{(\alpha, I+\langle x_1-\alpha\rangle)\;:\;\alpha\in S\}$
where $S\subseteq \K$ is infinite. In this case, only varieties with no irreducible component included in
a hyperplane $\{x_1=\alpha\}$ can be reconstructed. These varieties, as we show in section~\ref{sec2}, correspond
exactly with those given by ideals in good position according to the following definition: 

\begin{defn}
Let $I\subseteq \Kvars$ be an ideal. We say that $I$ is in good position geometrically (with
respect to the variable $x_1$) if
$\sqrt{I}=\bigcap_{i=1}^rP_i$ for some prime ideals $P_i$ such that $P_i\cap \K[x_1]=\{0\}$. Similarly,
we say that $I$ is in good position algebraically (w.r.t.~$x_1$) if $I=\bigcap_{i=1}^rQ_i$ for some primary ideals
$Q_i$ such that $Q_i\cap \K[x_1]=\{0\}$.
\end{defn}

The notions defined above seem to be very restrictive, but this is not necessarily the case: any ideal
whose variety has no zero-dimensional component can be rotated with a suitable linear change of variables
in such a way that the resulting ideal is in good position geometrically. Similarly, ideals with no
embedded zero-dimensional component can be put in good position algebraically through a linear change
of coordinates.

\begin{thm}\label{thm-2}
Let $I\subseteq \Kvars$ be an ideal and let $S\subseteq \K$ be an infinite set. Then:
\begin{itemize}
\item[(a)] $I$ is in good position geometrically (w.r.t.~$x_1$)
 $\displaystyle\Longrightarrow\sqrt{I}=\bigcap_{\alpha\in S}\sqrt{I+\langle x_1-\alpha\rangle}$.
\item[(b)] $I$ is in good position algebraically (w.r.t.~$x_1$)
 $\displaystyle\Longrightarrow I=\bigcap_{\alpha\in S}I+\langle x_1-\alpha\rangle$.
\end{itemize}
\end{thm}

For radical ideals $I\subseteq \Kvars$ in good position geometrically,  we can show
that $I=\bigcap_{\alpha\in S}I+\langle x_1-\alpha\rangle$ for any infinite set $S\subseteq \K$,
using a similar argument as in Eq.(\ref{eq_all}):
\begin{equation}\label{eq_inf}
I\subseteq\bigcap_{\alpha\in S}I+\langle x_1-\alpha\rangle=\bigcap_{\alpha\in S}\sqrt{I}+\langle x_1-\alpha\rangle
\subseteq\bigcap_{\alpha\in S}\sqrt{I+\langle x_1-\alpha\rangle}\;\overset{\ref{thm-2}(a)}{=}\;\sqrt{I}=I. 
\end{equation}

Finally, our third result studies the possibility of reconstructing a variety (or an ideal) from
finitely many cross sections.

\begin{thm}\label{thm-3}
Let $I=\langle f_1,\ldots,f_r\rangle\subseteq \Kvars$ be an ideal, let $S\subseteq \K$ be a finite set,
and let $f\in \Kvars$ with $\deg(f)\leq d$. Let $\delta=\max\{\deg(f_i)\,:\,i=1,\ldots,r\}$.
\begin{itemize}
\item[(a)] $I$ is in good position geometrically (w.r.t.~$x_1$) and $|S|>(d+1)\deg(V(I))$, then:
 $$f\in\sqrt{I} \Longleftrightarrow f\in\sqrt{I+\langle x_1-\alpha\rangle}\;\forall\,\alpha\in S,$$
where $\deg(V(I))$ is the maximum of the degrees of the irreducible components of $V(I)$.
\item[(b)] $I$ is in good position algebraically (w.r.t.~$x_1$) and
 $$|S|>\left(\left(d+2(\delta r)^{2^{n-1}}\right)^n+1\right)\max\{d,\delta \},$$
then:
 $$f\in I \Longleftrightarrow f\in I+\langle x_1-\alpha\rangle\;\forall\,\alpha\in S.$$
\end{itemize}
\end{thm}

Theorem~\ref{thm-3}(b) can be written as:
$$I\cap\{f\,:\,\deg(f)\leq d\} = \bigcap_{\alpha\in S}(I+\langle x_1-\alpha\rangle)\cap\{f\,:\,\deg(f)\leq d\},$$
where $S\subseteq \K$ has at least $\left(\left(d+2(\delta r)^{2^{n-1}}\right)^n+1\right)\max\{d,\delta \}+1$ elements.
In this formulation,
both sides of the equality are $\K$-vector spaces of dimension bounded by $\binom{d+n}{n}$, and in the case where
$d=\delta $, they include the generators of $I$. In particular, it is possible to compute generators of $I$ as the
base of the $\K$-vector space $\bigcap_{\alpha\in S}(I+\langle x_1-\alpha\rangle)\cap\{f\,:\,\deg(f)\leq \delta \}$
when $$|S|>\left(\left(\delta +2\left(\delta \binom{\delta +n}{n}\right)^{2^{n-1}}\right)^n+1\right)\delta .$$ The same conclusion is achieved with
the more simple bound $|S|\geq (\delta +n)^{(n+1)^22^n}$ that overestimates the bound above while keeping its order of
magnitude. It should be noted that, when the number of variables $n$ is fixed, the bound depends polynomially in $\delta$.

\bigskip

Theorem~\ref{thm-3} can be used to reduce the problem of the ideal membership~\cite{Mayr} (for ideals with no
zero-dimensional components), to several instances of the same problem with one variable less. The
idea is to perform first a linear change of coordinates to put the ideal in good position, and then
use the theorem to reduce the problem to a large enough number of cross sections. In the geometric
case, one can easily check whether a polynomial $f$ of degree $d$ vanishes on a given algebraic variety
$V$, by just testing if $f$ vanishes at $(d+1)\deg(V)$ cross sections of $V$.

\bigskip

In~\cite{BJK}, the authors prove that the ideal $I(V)$ of a smooth irreducible variety $V$ is generated by
polynomials of degree bounded by $\deg(V)$. They also provide a probabilistic method to compute those generators.
Theorem~\ref{thm-2}(a), can be used as an alternative procedure to compute the generators of $I(V)$, by iteratively
reducing the number of variables and the dimension of $V$, until we get to a zero-dimensional variety, where we
can use~\cite{BW96} or~\cite{KL91}. At each iteration we change the problem by $(\deg(V)+1)^2$ problems in one
variable less.

\bigskip

In the case of principal ideals, we obtained a much better bound, as shown in the following theorem.

\begin{thm}\label{thm-4}
Let $I= \langle f \rangle \subseteq \Kvars$ be a principal ideal generated by a non-zero polynomial of degree at most~$d $.
Assume that $f\not\in\K[x_1]$.
Let $I_k=I+\langle x_1-\alpha_k\rangle $ for $k=1,\ldots,2d $, where $\alpha_1,\ldots,\alpha_{2d }\in\K$ are
pairwise distinct. Then, the ideal $I$ can be uniquely reconstructed from the pairs $(\alpha_k,I_k)$.
\end{thm}

Note that the information that $I$ is principal has to be known a priori. We have also found (see Example~\ref{ex1})
two ideals $I,J\subseteq\C[x,y]$, generated by polynomials of degree $d $, and $2d -1$ points $\alpha_1,\ldots,\alpha_{2d -1}\in\C$,
such that $I+\langle x-\alpha_i\rangle=J+\langle x-\alpha_i\rangle$ for all $i=1,\ldots,2d -1$. This shows that
the bound of Theorem~\ref{thm-4} can not be improved.

\section{Interpolation of ideals and algebraic varieties}\label{sec2}

\begin{prop}\label{prop1}
Let $I\subseteq\Kvars$ be a radical ideal. Then
$$I=\bigcap_{\alpha\in\K}I+\langle x_1-\alpha\rangle.$$
\end{prop}
\begin{proof}
($\subseteq$): Trivial.
($\supseteq$): Let $f\in \bigcap_{\alpha\in\K}I+\langle x_1-\alpha\rangle$ and $p=(p_1,\ldots,p_n)\in V(I)$. Since $f\in I+\langle x_1-p_1\rangle$, there are $g\in I$ and $q\in\Kvars$ such that $f=g+(x_1-p_1)q$. Therefore
$f(p)=g(p)+(p_1-p_1)q(p)=0$. This implies that $f\in I(V(I))=\sqrt{I}=I$.
\end{proof}

Proposition~\ref{prop1} is the geometric analogue of slicing an algebraic variety and
then putting all these slices together. The same technique can be used to prove Theorem~\ref{thm-1}(a),
which is slightly stronger than Proposition~\ref{prop1}, since the ideal $\sqrt{I+\langle x_1-\alpha\rangle}$
contains the ideal $I+\langle x_1-\alpha\rangle$ for all $\alpha\in\K$.

\begin{proof}[Proof of Theorem~\ref{thm-1}(a).]
($\subseteq$): Trivial.
($\supseteq$): Let $f\in \bigcap_{\alpha\in\K}\sqrt{I+\langle x_1-\alpha\rangle}$ and $p=(p_1,\ldots,p_n)\in V(I)$.
There exists $k\geq 1$ such that $f^k\in I+\langle x_1-p_1\rangle$. This means that $f^k$ can be written as
$f^k=g+(x_1-p_1)q$ for some $g\in I$ and $q\in\Kvars$. Evaluating at the point~$p$, we get
$f^k(p)=g(p)+(p_1-p_1)q(p)=0$, and then $f(p)=0$. This implies that $f\in I(V(I))=\sqrt{I}$.
\end{proof}

\begin{lemma}\label{lemma2}
Let $I\subseteq\Kvars$ be an ideal and let $f,g\in\K[x_1]$ with
$\gcd(f,g)=1$. Then $$(I+\langle f\rangle )\cap(I+\langle g\rangle )=I+\langle fg\rangle.$$
\end{lemma}
\begin{proof}
($\supseteq$): Trivial.
($\subseteq$): Take $h\in (I+\langle f\rangle)\cap(I+\langle g\rangle)$. We can write
$h=h_1+ff'=h_2+gg'$ with $h_1,h_2\in I$ and $f',g'\in \Kvars$. Let $a,b\in \K[x_1]$ be
polynomials such that $af+bg=1$. Since $ff'=h_2-h_1+gg'$, then $aff'=a(h_2-h_1)+agg'$, and
also $f'=a(h_2-h_1)+g(ag'+bf')$. This implies that $ff'\in I+\langle fg\rangle$, and since
$h_1\in I$, we conclude that $h=h_1+ff'\in I+\langle fg\rangle$.
\end{proof}

Lemma~\ref{lemma2} allows us to rewrite Proposition~\ref{prop1} as follows:
$$I\subseteq\Kvars\;\mbox{radical}\;\Longrightarrow
I=\!\!\!\!\!\!\bigcap_{\doslineas{p\in\K[x_1]\setminus\{0\}}{{\text{squarefree}}}}\!\!\!\!\!\!I+\langle p\rangle.$$

Proposition~\ref{prop1} does not work for general ideals. For instance,
the ideal $I=\langle x_1^2x_2\rangle$ and $J=\langle x_1x_2\rangle$
satisfy $I+\langle x_1-\alpha\rangle = J+\langle x_1-\alpha\rangle$ for
all $\alpha\in\K$, but $I\neq J$. Theorem~\ref{thm-1}(b) shows that this problem
can be avoided by considering arbitrarily large powers of $x_1-\alpha$.

\begin{proof}[Proof of Theorem~\ref{thm-1}(b).]
Let $\mathcal{P}$ denote the set of non-zero polynomials in $\K[x_1]$.
By Lemma~\ref{lemma2} it is enough to show that $I=\bigcap_{p\in\mathcal{P}}I+\langle p\rangle$.
We show first that the we can reduce the proof to the case where $I$ is a primary ideal.
Indeed, if $I=Q_1\cap\cdots\cap Q_r$ with $Q_i$ primary ideals, then 
$$I\subseteq\bigcap_{p\in\mathcal {P}}I+\langle p\rangle
 =\bigcap_{p\in\mathcal {P}}((Q_1\cap\dots\cap Q_r)+\langle p\rangle)\subseteq$$
$$\subseteq \bigcap_{p\in\mathcal {P}}((Q_1+\langle p\rangle)\cap\dots\cap (Q_r+\langle p\rangle))=
            \bigcap_{i=1}^r\bigcap_{p\in\mathcal {P}}(Q_i+\langle p\rangle)=\bigcap_{i=1}^r Q_i=I.$$

Besides, if there is a non-zero polynomial $q$ in $I$ pure in $x_1$, then it is clear that
$$I\subseteq\bigcap_{p\in\mathcal {P}} I+\langle p\rangle\subseteq I+\langle q\rangle=I.$$ This
reduces the proof to the case of primary ideals $I$ such that $I\cap\K[x_1]=\{0\}$.

Take $I$ a primary ideal with $I\cap \K[X_1]={0}$. Let $f\in I+\langle p\rangle$.
For all $p\in\mathcal{P}$ we can write $f=f_p+pg_p$ with $f_p\in I$ and $g_p\in \Kvars$.
Now we compare the two representations of $f$ with subindices $p$ and $pq$ for $p,q\in\mathcal{P}$.
We have that $f=f_p+pg_p=f_{pq}+pqg_{pq}$. This implies that $p(g_p-qg_{pq})\in I$ and, since $p\notin\sqrt{I}$,
we get $g_p-qg_{pq}\in I$, and also that $g_p\in I+\langle q\rangle$.
Write $J=\bigcap_{p\in\mathcal {P}}I+\langle p\rangle$.
The previous discussion proves that $J\subseteq\bigcap_{p\in\mathcal{P}}(I+\langle p\rangle J)$, and since
the other inclusion is trivial, we obtain:
\begin{equation}\label{eq1}
 J=\bigcap_{p\in\mathcal {P}}(I+\langle p\rangle J)
\end{equation}
Now we localize Eq.(\ref{eq1}) at the maximal ideal
$\M=\langle x_1-\alpha_1,\ldots,x_n-\alpha_n\rangle\in \Kvars$.

$$J_\M=\left( \bigcap_{p\in\mathcal{P}}(I+\langle p\rangle J)\right)_\M
      \subseteq \bigcap_{p\in\mathcal {P}}(I+\langle p\rangle J)_\M
      \subseteq \bigcap_{p\in\mathcal {P}}(I_\M+\langle p\rangle J_\M) \subseteq J_\M.$$

Now we have 
$J_\M=\bigcap_{p\in\mathcal {P}}I_\M+\langle p\rangle J_\M$ as $\Kvars_\M$-modules. This intersection
can be rewritten as:
$$J_\M=\underbrace{\left(\bigcap_{\doslineas{p\in\mathcal {P}}{p(\alpha_1)=0}}I_\M+\langle p\rangle J_\M\right)}_{J'}\cap
  \underbrace{\left(\bigcap_{\doslineas{p\in\mathcal {P}}{p(\alpha_1)\neq 0}}I_\M+\langle p\rangle J_\M\right)}_{J''}.$$
For any $p\in\mathcal{P}$ such that $p(\alpha_1)\neq 0$, we have that $\langle p\rangle=\langle 1\rangle$ in
$\Kvars_\M$ and consequently $I_\M+\langle p\rangle J_\M=J_\M$. Therefore $J''=J_\M$.
For any $p\in\mathcal{P}$ with $p(\alpha_1)=0$, we have that $\langle p\rangle \subseteq \langle x_1-\alpha_1\rangle$,
and therefore $J'\subseteq I_\M+\langle x_1-\alpha_1\rangle J_\M$. All together, this shows that
$J_\M=J'\cap J''\subseteq I_\M+\langle x_1-\alpha_1\rangle$, and by Nakayama's Lemma $J_\M=I_\M$.
Since this is true for any maximal ideal $\M$, it follows from the global-local principle that $I=J$.
\end{proof}

Theorem~\ref{thm-1}(b) is the algebraic counterpart of the more geometric intuitive
Proposition~\ref{prop1} and Theorem~\ref{thm-1}(a). These results show that ideal reconstruction
is possible if we are given all the cross sections. Indeed, it is possible to recover
ideals (with no vertical embbedded components) with infinitely many sections, as we show
below. The extra assumption is necessary, as shown by the ideals
$I=\langle (x+y)^2, (x+y)x\rangle = \langle x+y\rangle \cap \langle x,y\rangle$ and
$J=\langle x+y\rangle$ which satisfy $I+\langle x-\alpha\rangle = J+\langle x-\alpha\rangle = \langle x+y, x-\alpha\rangle$
for all $\alpha\neq 0$, but $I\neq J$. The problem in this example comes from the embedded component $\{(0,0)\}$ of
$I$, corresponding to the primary ideal $\langle x,y\rangle$, that is invisible to all the vertical planes $\{x=\alpha\}$
with $\alpha\neq 0$.

\begin{lemma}\label{alg-lin}
Let $A(t)\in \K[t]^{N\times M}$ and $b(t)\in \K[t]^{N\times 1}$.
\begin{itemize}
\item[1.] If $(A(t)|b(t))$ is incompatible in $\K(t)$ then $(A(\alpha)|b(\alpha))$ is compatible in $\K$ for only a finite number of $\alpha\in \K$.
\item[2.] If $(A(t)|b(t))$ is compatible in $\K(t)$ then $(A(\alpha)|b(\alpha))$ is compatible in $\K$ for $\alpha\in \K$ but a finite number.
\item[3.] Assume that $\deg(A_{ij}),\deg(b_i)\leq d$ for all $1\leq i\leq N$ and $1\leq j\leq M$. Let $S\subseteq\K$ with $|S|>d\max\{N, M+1\}$.
Then $(A(t)|b(t))$ is compatible if and only if $(A(\alpha)|b(\alpha))$ is compatible for all $\alpha\in S$.
\end{itemize}
\end{lemma}
\begin{proof}
The rank of any matrix with coefficients in $\K[t]$ is the size of the largest
submatrix with non-zero determinant. Since the determinant of that submatrix is a polynomial
in $t$, its evaluation at $\alpha$ is non-zero for almost every $\alpha\in\K$.
The first two statements follow immediately from that remark and the fact
that a system $(A|b)$ is compatible if and only if $\rank(A|b)=\rank(A)$.
For the last item, note that the degree of the determinant of any square submatrix of $(A(t)|b(t))$ has degree
at most $d\max\{N, M+1\}$.
\end{proof}

\begin{thm}\label{cota}
Let $I=\langle f_1,\ldots,f_r\rangle\subseteq\Kvars$ be an ideal with $\deg(f_i)\leq \delta$ for $i=1,\ldots,r$,
and let $f\in I$. Then there exists $g_1,\ldots,g_r\in\Kvars$ such that $f=g_1f_1+\cdots+f_rg_r$ and
$\deg(g_i)\leq\deg(f)+2(r\delta)^{2^{n-1}}$.
\end{thm}
\begin{proof}
See~\cite[Anwendung von Satz.~3]{Her26}.
\end{proof}

\begin{lemma}\label{lemainter}
Let $I\subseteq\Kvars$ be a primary ideal with $I\cap\K[x_1]=\{0\}$, then
$$(I+\langle x_1-t \rangle )\cap\Kvars=I$$
where $I+\langle x_1-t \rangle$ is regarded as an ideal of $\K(t)[x_1,\ldots,x_n]$.
\end{lemma}
\begin{proof}
($\supseteq$): Trivial.
($\subseteq$): Assume that $I=\langle f_1,\ldots,f_r\rangle$ with $f_i\in\Kvars$.
Take $f\in \K[x_1,\ldots, x_n]$ and suppose that it can be written as $f=f_1g_1+\ldots+f_rg_r+(x_1-t)g$
with $g_i\in \K(t)[x_1,\ldots, x_n]$. Clearing denominators by multiplying by $\omega(t)\in \K[t]$,
gives $\omega(t)f=f_1\bar{g_1}+\ldots+f_r\bar{g_r}+(x_1-t)\bar{g}$ where
$\bar{g_1},\ldots,\bar{g_r},\bar{g}\in \K[t,x_1,\ldots,x_n]$. Since $f_1,\ldots, f_r,f$ do not
involve the variable $t$, substituting $t=x_1$, gives $\omega(x_1)f\in I$. Besides,
$\omega(x_1)\notin\sqrt{I}$ because $I\cap \K[x_1]={0}$. Since $I$ is primary, we conclude that
$f\in I$.
\end{proof}

We start with a simplified version of Theorem~\ref{thm-2}(b) for primary ideals.

\begin{thm}\label{thm-S-inf}
Let $I\in\Kvars$ be a primary ideal with $I\cap \K[x_1]=\{0\}$ and let $S\subseteq\K$ be an infinite set, then
$$I=\bigcap_{\alpha\in S}I+\langle x_1-\alpha\rangle.$$
\end{thm}
\begin{proof}
Assume that $I=\langle f_1,\ldots,f_r\rangle$ with $f_i\in\Kvars$ and take $f\in\Kvars$.
Let $C=\deg(f)+2((r+1)\delta)^{2^{n-1}}$ where $\delta=\max\{1,\deg(f_1),\ldots,\deg(f_r)\}$. For a
given $\alpha\in\K$, we have that $f\in I+\langle x_1-\alpha\rangle$ if and only if
there exist $g_1,\ldots,g_r,g\in\Kvars$ with degree bounded by $C$ such that
$f=f_1g_1+\cdots+f_rg_r+(x_1-\alpha)g$, by Theorem~\ref{cota}. This is a linear system of
equations with coefficients that depend polynomially in $\alpha$. By Lemma~\ref{alg-lin},
if this system is
compatible for an infinite number of $\alpha$, then it is compatible in $\K(\alpha)$
where $\alpha$ is regarded as an indeterminate. Conversely, if the system is incompatible
for infinitely many values of $\alpha$, then it is also incompatible in $\K(\alpha)$.
All together this says that:
$$f\in\bigcap_{\alpha\in S}I+\langle x_1-\alpha\rangle\;\Longleftrightarrow\;
  f\in I+\langle x_1-t\rangle\subseteq\K(t)[x_1,\ldots,x_r].$$
We conclude immediately the proof by using Lemma~\ref{lemainter}.
\end{proof}

At this point we have all the tools needed to show the main result of this section.

\begin{proof}[Proof of Theorem~\ref{thm-2}(b).]
($\subseteq$): Trivial.
($\supseteq$): Assume that $I=\bigcap_{i=1}^r Q_i$ with $Q_i$ primary and $Q_i\cap \K[x_1]=\{0\}$. We have that: 
\begin{eqnarray*}
\bigcap_{\alpha\in S} I+ \langle x_1-\alpha \rangle & = & \bigcap_{\alpha\in S}\left[ \left( \bigcap_{i=1}^r Q_i\right)+ \langle x_1-\alpha \rangle \right]\subseteq\\
& \subseteq & \bigcap_{\alpha\in S} \bigcap_{i=1}^r \left(Q_i+ \langle x_1-\alpha \rangle \right)=\\
&=&\bigcap_{i=1}^r\bigcap_{\alpha\in S} \left(Q_i+ \langle x_1-\alpha \rangle \right).
\end{eqnarray*}
By Theorem~\ref{thm-S-inf}, the last term of the previous chain of inclusions is
equal to $\bigcap_{i=1}^r Q_i~=~I$.
\end{proof}

\begin{thm}\label{cota-rad}
Let $I=\langle f_1,\ldots,f_r\rangle\subseteq\Kvars$ be an ideal with $\deg(f_i)\leq \delta$ for $i=1,\ldots,r$,
and let $f\in\sqrt{I}$ with $\deg(f)\leq \delta$. Then there exists $g_1,\ldots,g_r,g\in\K[t,x_1,\ldots,x_n]$ such that
$1=g_1f_1+\cdots+f_rg_r+(1-tf)g$ with $\deg(g_i)$ and $\deg(g)$ bounded above by $\max\{3,\delta+1\}^{n+1}$.
\end{thm}
\begin{proof}
See \cite[Thm.~1.5]{Kollar}. See also \cite{FitchGalligo}, \cite[Thm.~1.1]{Jel05} and~\cite[Thm.~1]{Sombra}
for an alternative proof.
\end{proof}

A similar conclusion to Theorem~\ref{cota-rad} can be obtained from Theorem~\ref{cota}, but with worse bound.
Although any finite bound would have been enough to show the following theorem, we included it here since it gives
an idea of the size of the linear algebra problem involved in the proof.

\begin{proof}[Proof of Theorem~\ref{thm-2}(a).]
$(\subseteq)$: Trivial. $(\supseteq)$:
Assume that $I=\langle f_1,\ldots,f_r \rangle$ with $f_i\in\Kvars$.
Take $f\in \bigcap_{\alpha\in S}\sqrt{I+ \langle x_1-\alpha \rangle}$ and let $\delta=\max\{\deg(f),\ \deg_{i=1,\ldots,r} (f_i)\}$.
Define $C=\max\{3,\delta+1\}^{n+1}$ the constant of Theorem~\ref{cota-rad}.
For all $\alpha \in S$, the linear system $1=f_1g_1+\ldots+f_rg_r+(x_1-\alpha)h+ (1-tf)g$ with $\deg(g),\deg(h),\deg(g_i)\leq C$
is compatible in $\K$, i.e. there are $g_1,\ldots,g_r,g,h\in \K[x_1,\ldots,x_n]$, that depend on $\alpha$, such that
$1=f_1g_1+\ldots+f_rg_r+(x_1-\alpha)h+(1-tf)g$. By Lemma~\ref{alg-lin}, the system is also compatible over $\K(\alpha)$ where
$\alpha$ is regarded as an indeterminate. This means that, in the expression above, $g_1,\ldots,g_r,h,g$ can be taken in
$\K(\alpha)[t,x_1,\ldots,x_n]$.
Multiplying by $\omega(\alpha)$ in order to clear denominators, we get
$$\omega(\alpha)=f_1\bar{g}_1+\ldots+g_r\bar{g}_r+(x_1-t)\bar{h}+(1-tf)\bar{g}$$
where $\bar{g}_1,\ldots,\bar{g}_r,\bar{h},\bar{g}\in\K[\alpha,t,x_1,\ldots,x_n]$.
Susbtituting $\alpha=x_1$, we get
$$\omega(x_1)=f_1\tilde{g}_1+\ldots+f_r\tilde{g}_r+(1-tf)\tilde{g}$$
where $\tilde{g}_1,\ldots,\tilde{g}_r,\tilde{g}\in\K[t,x_1,\ldots,x_n]$.
Finally, substituting $t=\frac{1}{f}$ and removing denominators by multiplying by a large enough power of $f$,
we obtain $f^N\omega(x_1)\in I$, which implies that $f\omega(x_1)\in\sqrt{I}$.
Since $\sqrt{I}= P_1\cap\cdots\cap P_s$ with $P_i$ prime and $P_i\cap\K[x_1]= {0}$, we have that
$\omega(x_1)\notin P_i$ and therefore $f\in P_i$ for all $i$.
\end{proof}

\section{Recovering an ideal from finitely many cross sections}

Let $I=\langle f_1, \ldots, f_r\rangle\subseteq \Kvars$ be an ideal and let $\alpha\in\K$.
Throughout this section we will use the following notation:
$$I|_{x_1=\alpha}=\langle f_1|_{x_1=\alpha}, \ldots, f_r|_{x_1=\alpha}\rangle\subseteq \K[x_2,\ldots, x_n].$$

\begin{thm}\label{sueno2}
Let $I\subseteq\Kvars$ be an ideal such that:
\begin{itemize}
 \item $V(I)$ is equidimensional.
 \item $V(I)$ has no irreducible component contained in a hyperplane $\{x_1=\alpha\}$.
\end{itemize}

Let $f\in\Kvars$ with $\deg(f)\leq d$. Then

$$ f\in \sqrt{I}\Leftrightarrow f|_{x_1=\alpha}\in\sqrt{I|_{x_1=\alpha}}$$
for all $\alpha\in S$ with $|S|>(d+1)\deg(V(I))$.
\end{thm}

\begin{proof}
($\Rightarrow$): Trivial. ($\Leftarrow$): We proceed by induction in $\dim(V(I))$.

 \begin{itemize}
  \item Case $\dim(V(I))=1$: We have that $V(I)$ is a union of irreducible curves $\CC_1\cup\cdots\cup\CC_m$.
Our assumptions imply that $f$ vanishes at $V(I)\cap\{x_1=\alpha\}$ for all $\alpha\in S$, and in
particular, $f$ vanishes at $\CC_i\cap\{x_1=\alpha\}$ for all $\alpha\in S$ and $i=1,\ldots,m$.
We know that
$|\CC_i\cap \{x_1= \alpha\}|\geq 1$ for all $\alpha$ except maybe for those values where the compactification
of $\CC_i$ in $\PP^n$ intersects the hyperplane $\{x_1=\alpha\}$ at infinity. Since there are at most
$\deg(V(I))$ of such points, we have $|V(f)\cap\CC_i|>\deg(V(I))d$. By Bezout's Theorem (see~\cite[Thm.~2.1]{Vogel}),
we have that either $|V(f)\cap\CC_i|\leq \deg(V(f))d$ or $f$ vanishes at $\CC_i$. We have shown above that the former
cannot happen, so we conclude that $f\in I(\CC_i)$ for all $i=1,\ldots,m$. Therefore $f\in\sqrt{I}$.

\bigskip

\item Case $\dim(V(I))=e>1$: Assume the theorem is true for $\dim(V(I))\leq e-1$.
Without loss of generality we can assume that, after a suitable linear change of coordinates, there exist an infinite
set $\Omega\subseteq\K$ such that the ideals $I|_{x_2=\beta}$ satisfy:
\begin{itemize}
 \item $\deg(V(I))=\deg(V(I|_{x_2=\beta}))$,
 \item $V(I|_{x_2=\beta})$ is equidimensional,
 \item $\dim (V(I|_{x_2=\beta}))=\dim V(I)-1\geq 1$,
 \item $V(I|_{x_2=\beta})$ has no irreducible component contained in a hyperplane $\{x_1=\alpha\}$,
\end{itemize}
for all $\beta\in\Omega$.
In particular, the ideals $I|_{x_2=\beta}$ satisfy the induction hypothesis with
$\dim(V(I|_{x_2=\beta}))=e-1$.
If $f|_{x_1=\alpha}\in\sqrt{I|_{x_1=\alpha}}$ for $\alpha\in S$ with $|S|>(d+1)\deg(V(I))$, then
we also have that $f|_{x_1=\alpha,\ x_2=\beta}\in\sqrt{I|_{x_1=\alpha,\ x_2=\beta}}$.
Consequently, $f|_{x_2=\beta}\in\sqrt{I|_{x_2=\beta}}$ for all $\beta\in \Omega$.
By Theorem~\ref{thm-2}(a), we conclude that $f\in \bigcap_{\beta\in S}\sqrt{I+\langle x_2-\beta\rangle}=\sqrt{I}$.
\end{itemize}
\end{proof}

Now Theorem~\ref{thm-3}(a) follows immediately as a corollary.

\begin{proof}[Proof of Theorem~\ref{thm-3}(a).]
Our assumptions imply that $$V(I)=V=V_1\cap V_2\cap \ldots V_e$$ where $e=\dim V$ and $V_i$ are equidimensional varieties
of dimension $i$, neither of them included in a hyperplane $\{x_1=\alpha\}$. The following diagram holds:

 \vspace{0.25cm}\hspace{4cm}
 \xymatrix@C=1pc@R=1pc
{f|_{V}\equiv 0 \ar@{<=>}[dd]_{\forall i} \ar@{<=>}[rr]
 & & f|_{V\cap \{x_1=\alpha\}}\equiv 0\\ \\
 f|_{V_i}\equiv 0\ar@{<=>}[rr]^{(\ast)\;\quad}& & f|_{V_i\cap \{x_1=\alpha\}}\equiv 0\ar@{<=>}_{\forall i\atop\forall\alpha\in S}[uu]}
\vspace{0.25cm}

The arrow $(\ast)$ follows from Theorem~\ref{sueno2}. By the Nullstellensatz, the arrow on top is equivalent to say that
$f\in I\Longleftrightarrow f\in\sqrt{I+\langle x_1-\alpha\rangle}\;\forall\,\alpha\in S$.
\end{proof}

In the algebraic case, we proceed as in the proof of Theorem~\ref{thm-2}(b), but keeping track of the bounds of
the degrees.

\begin{proof}[Proof of Theorem~\ref{thm-3}(b).]
Assume that $I=\bigcap_{i=1}^lQ_i$ where $Q_i$ are primary ideals with $Q_i\cap\K[x_1]=\{0\}$.
By Theorem~\ref{cota}, we have that $f\in I$ if and only if there exist $g_1,\ldots,g_r\in\Kvars$ with
$\deg(g_i)\leq d+2(\delta r)^{2^{n-1}}$ such that $f=f_1g_1+\cdots+f_rg_r$. This equation can be regarded as a
linear system of equations in $\K[x_1]$:
$$f\in I\Longleftrightarrow A(x_1) G = b(x_1),$$
where $A(x_1)$ and $b(x_1)$ are matrices whose entries are coefficients of $f_1,\ldots,f_r$ and $f$
respectively. The unknowns are the coefficients of $g$, represented by the vector $G$.
Therefore $f\in I$ if and only if that system is compatible in $\K(x_1)$. By Lemma~\ref{alg-lin}, that
system is compatible if and only if the system $(A(\alpha)|b(\alpha))$ is compatible for $\alpha\in S$
with $|S|>\max\{d,\delta \}\max\{{\rm rows}(A), {\rm cols}(A)+1\}$. Using Theorem~\ref{cota} again, each of those systems are
compatible if and only if $f|_{x_1=\alpha}\in I|_{x_1=\alpha}$, or equivalently, $f\in I+\langle x_1-\alpha\rangle$. The result
follows by counting the number of rows and columns of $A$: ${\rm rows}(A)\leq d^n$ and
${\rm cols}(A)\leq\left(d+2(\delta r)^{2^{n-1}}\right)^n$.
\end{proof}

\section{Principal ideals}

\begin{remark}\label{remfacil}
 Let $I=\langle f\rangle\subseteq\Kvars$ be a principal ideal, and let $J=I+\langle x_1-\alpha\rangle$ with $\alpha\in\K$.
 Then $J\cap \K[x_2,\ldots,x_n] = \langle f(\alpha,x_2,\ldots,x_n)\rangle$ in $\K[x_2,\ldots,x_n]$.
\end{remark}

\begin{proof}[Proof of Theorem~\ref{thm-4}.]
Throughout this proof we will write $x=x_1$ and $y=(x_2,\ldots,x_n)$. We will order the monomials in $y$
using the graded lexicographic order $x_2>x_3>\cdots>x_n$.
Let $$f=\sum_{i\,:\,|i|\leq d}a_i(x)y^i$$ where $i=(i_2,\ldots,i_n)$ and $y^i=x_2^{i_2}\cdots x_n^{i_n}$.
By Remark~\ref{remfacil}, for any $k=1,\ldots,2d$, we have that
$$I_k\cap\K[y]=\langle f,x-\alpha_k \rangle\cap\K[y]= \langle g_k\rangle$$
where $g_k=\lambda_k f(\alpha_k, y)$ has leading coefficient $1$ and $\lambda_k\in\K^\ast$.
The following identities show that it is possible to recover ${\rm multideg}_y(f)$ from the $g_k$:
\begin{eqnarray*}
e = {\rm multideg}_y(f) & = & \max\{i\,:\,a_i(x)\neq 0\} = \\
& \overset{(\ast)}{=} & \max\{i\,:\,a_i(\alpha_k)\neq 0\;\text{for some}\;k\} = \\
& = & \max_{k=1,\ldots,2d}\left(\max\{i\,:\,a_i(\alpha_k)\neq 0\}\right) = \\
& = & \max_{k=1,\ldots,2d}{\rm multideg}_y(f(\alpha_k,y))=\\
& = & \max_{k=1,\ldots,2d}{\rm multideg}_y(g_k).
\end{eqnarray*}
$(\ast)$: This equality is true since $\deg(a_i)\leq d-|i|<2d$.

Now we know that $f=\sum_{i\leq e}a_i(x)y^i$ with $a_e\neq 0$. Since $f\not\in\K[x]$, then $|e|\geq 1$.
The polynomial $a_e(x)$ vanishes in exactly $r\leq d-|e|$ points in $\{\alpha_1,\ldots,\alpha_{2d}\}$.
Without loss of generality we can assume that $a_e(\alpha_{2d-r+1})=\cdots=a_e(\alpha_{2d})=0$, i.e.
\begin{equation}\label{formula1}
 a_e(x)=\tilde{a}_e(x)\cdot\prod_{l=2d-r+1}^{2d}(x-\alpha_l)
\end{equation}
where $\tilde{a}_e\in\K[x]$ has degree at most $d-|e|-r$. Since the polynomials $g_k=\lambda_k f(\alpha_k,y)$
have leading coefficient $1$, then $\lambda_k=\frac{1}{a_e(\alpha_k)}$ for $k=1,\ldots,2d-r$. In particular,
the coefficients of $g_k$, which are all known, are equal to $\frac{a_i(\alpha_k)}{a_e(\alpha_k)}$ for
$1\leq k\leq 2d-r$ and $0\leq i\leq e$. Combining this with Eq.(\ref{formula1}), we can obtain the following
fractions:
$$\frac{a_i(\alpha_k)}{\tilde{a}_e(\alpha_k)} = \frac{a_i(\alpha_k)}{a_e(\alpha_k)}
  \cdot\prod_{l=2d-r+1}^{2d}(\alpha_k-\alpha_l).$$
Since $\deg(a_i)\leq d-|i|$ and $\deg(\tilde{a}_e)\leq d-|e|-r$, it is possible to reconstruct the rational
function $\frac{a_i(x)}{\tilde{a}_e(x)}$ from the $2d-r\geq 2d-|i|-|e|-r+1$ points $\alpha_1,\ldots,\alpha_{2d-r}$
using rational interpolation.
\end{proof}

The following example shows that $2d-1$ cross sections are not enough.

\begin{example}\label{ex1}
Consider the polynomials $f=p(x)y+1$ and $g=a(x)y+b(x)$, where $p(x)=x^d$, $a(x)=-x^{d-1}+2^{2d-1}$, and $b(x)=x^d-1$.
The ideals $I=\langle f\rangle$ and $J=\langle g\rangle$ are both principal, generated by
polynomials of degree $d$, and we clearly have $I\neq J$. Let $\alpha_i=2\xi_{2d-1}^i$ where $\xi_{2d-1}\in\C$ is
a primitive $(2d-1)$-th root of unity. Let us see that the ideals $I+\langle x-\alpha_i\rangle$ and
$J+\langle x-\alpha_i\rangle$ are equal for $i=1,\ldots,2d-1$. Indeed, a simple computation shows that
$(2^d\xi_{2d-1}^{id}-1)f(\alpha_i,y)=g(\alpha_i,y)$.
\end{example}

\section*{Acknowledgements}

We thank Jessica Aliaga Lavrijsen, Jos\'e Ignacio Cogolludo-Agust\'\i n, \'Alvaro Lozano-Rojo and Silvia
Vilari\~no Fern\'andez for many valuable comments
and corrections.

\end{document}